%% file: main.tex
\documentclass{article}
\usepackage{graphicx} 
\usepackage{amsmath}
\usepackage{amsthm}
\usepackage{amsfonts}
\usepackage{amssymb}
\usepackage{commath}
\usepackage{xcolor}
\usepackage{authblk}

\input{symbols.tex}

\title{Counting Algebraic Integers of Bounded Height in Cyclotomic Fields}

\author[1]{Phillip Harris}
\affil[1]{University of Bonn}

\author[2]{John Yin}
\affil[2]{The Ohio State University}

\begin{document}

\maketitle

\begin{abstract}
   In this note, we fix a height bound $B$, and give a bound on the number of algebraic integers and units of absolute height at most $B$ in cyclotomic extensions $\bbQ[\zeta_q]$, where $q$ is a prime power. The bound is asymptotic in $q$.
\end{abstract}

\section{Introduction}

Let $K$ be a number field and let $L$ be a finite extension. We define the exponential Weil height relative to $K$ of an element of $[\alpha_0: \dots :\alpha_n] \in \bbP^n_K(L)$ by $$H([\alpha_0:\dots:\alpha_n]) = \left(\prod_{\nu \in \text{ places of }L} \max(|\alpha_0|_{\nu}, \dots, |\alpha_n|_{\nu})\right)^{1/[L:K]},$$ where $L$ can be taken to be any number field containing the $\alpha_i$. We also define the logarithmic Weil height relative to $K$ to be 
\[
h([\alpha_0:\dots:\alpha_n]) =
    \log(H([\alpha_0:\dots:\alpha_n])).
\]

Much of the literature on counting points of bounded height derives its motivation from the Batyrev-Manin conjecture, which fixes a base field $K$, and predicts the number of $K$ points of bounded height $B$, as $B$ grows. In this note, we are curious about fixing the height bound, and letting the field vary instead.

Previous work in this aspect includes \cite{DUBICKAS_2018}. Here, Dubickas showed the following. Let $N(d, B)$ be the number of $\alpha \in \bar \bbQ$ of degree at most $d$, with log Weil height at most $B$. Then, for each fixed $B > 0$, $\log N(d, B) \sim B d^2$ as $d \to \infty$. 

We choose to study the prime power cyclotomic fields $\bQ(
\zeta_q), q = p^r$ as $r \to \infty$. Our main theorem concerns units in the ring of integers. Define
\[
U(n,B) = \# \{\alpha \in \bbZ[\zeta_n]^{\times}: h(\alpha) < B\},
\]
and
\[
U_{\cyc}(n,B) = \# \{\alpha \in \bbZ[\zeta_n]^{\times}: h(\alpha) < B, \alpha \text{ is a cyclotomic unit}\}.
\]
Then, we will show
\begin{thm} \label{thm:main}
    Let $q=p^r$ be an odd prime power. For all $\epsilon > 0$ and $B>0$, the cyclotomic units in $\bQ(\zeta_q)$ satisfy
    \[
    U_{\cyc}(q,B) \ll q^3 \exp \left( C_\epsilon B^{1/3} q^{5/6+\epsilon} \right).
    \]
    Moreover, the full unit group of $\bQ(\zeta_q)$ satisfies
    \[
    U(q,B) \ll q^3 \exp \left( C_\epsilon (h_q^+ B)^{1/3} q^{5/6+\epsilon} \right),
    \]
    where $h_q^+$ is the class number of $\bQ(\zeta_q)^+$.
\end{thm}

While we do not have much control over $h^+_q$, numerical evidence suggests it is small. In particular, using Cohen-Lenstra heuristics, Buhler, Pomerance, and Robertson conjecture that $h_{p^r}^+=h_{p^{r+1}}^+$ for all but finitely many pairs $(p,r)$ \cite{BPR2004}. Thus we conjecturally have a subexponential bound for fixed $p$ as $r \to \infty$.

We also give an exponential upper bound for algebraic integers.

\begin{thm} \label{thm:integers}
Let $q=p^r$ be an odd prime power, let $D=\phi(q)$, and set
\[
I(q,B)=\#\{ \alpha \in \bZ[\zeta_q] : h(\alpha) \leq B \}.
\]
Then
\[
\log I(q,B) \ll D(B+1).
\]
\end{thm}

Finally, we observe that both the above quantities have a polynomial lower bound. We also note that $\bF_q(t)$ is a natural function field analogue for $\bQ(\zeta_q)$, and in this case, the height is just the degree in $t$, giving $\# \{\alpha \in \bbF_q[t]: h(\alpha) \leq B\} = q^{B+1}$. We might speculate that the true quantity of $\bZ[\zeta_q]$-points is roughly polynomial in $q$.

\subsection{Acknowledgments}
We thank Sun-Woo Park for many helpful discussions. PH is supported by Deutsche Forschungsgemeinschaft (DFG) Grant No. SFB-TRR 358/1 2023 — 491392403. JY is supported by National Science Foundation RTG grant DMS-2231565. AI was used for minor writing and proofreading help. 

\section{Proof}

The main geometric input is the following lattice point counting lemma of Davenport \cite{Davenport1951}. Let $\cP \subset \bR^D$ be a compact body. For each $0 \leq j \leq D$, set
\[
V_j(\cP) = \sum_{\substack{I \subset [D] \\ \abs{I} = j}} \Vol_j (\pi_I(\cP)),
\]
where $\pi_I$ is the projection to $\bR^I$ and we normalize $V_0(\cP)=1$. If every line parallel to a coordinate axis meets $\cP$ in at most $h$ intervals, and the same is true for every coordinate projection $\pi_I(\cP)$, then
\[
\left| \#(\bZ^D \cap \cP) - \operatorname{Vol}(\cP) \right|
\leq
\sum_{j=0}^{D-1} h^{D-j} V_j(\cP).
\]
We also record the standard volume formula for zonotopes (the projection of a parallepiped, equivalently, a Minkowski sum of line segments).  Let $M$ be a $d \times n$ matrix with columns $v_1,\dots,v_n \in \bR^d$, defining a zonotope
\[
Z=\left\{ \sum_{i=1}^n t_i v_i : 0 \leq t_i \leq 1 \right\}.
\]
Write $[M]_{I,J}$ for the determinant of the submatrix of $M$ with rows $I$ and columns $J$. Then
\[
\Vol_d(Z)
=
\sum_{\substack{I \subset [n] \\ \abs{I}=d}}
\abs{[M]_{[d],I}}.
\]
We will also use Jacobi's complementary minor identity: if $A$ is an invertible $D \times D$ matrix and $I,J \subset [D]$ satisfy $\abs{I}=\abs{J}$, then writing $I^c=[D] \setminus I$ and $J^c=[D] \setminus J$ one has
\[
[A]_{I,J}
=
(-1)^{\sum_{i \in I} i + \sum_{j \in J} j}
\det(A)[A^{-1}]_{J^c,I^c}.
\]

\subsection{Units}
Let $q=p^r$ with $p$ an odd prime. Let $E_q \subset \bZ[\zeta_q]^\times$ be the group of units. Let $V_q \subset \bQ(\zeta_q)^\times$ be the multiplicative subgroup generated by $\zeta_q$ and $1-\zeta_q^a$ for all $a \neq 0 \pmod q$. Then $C_q = E_q \cap V_q$ is called the group of \textit{cyclotomic units} of $\bQ(\zeta_q)$.
Let $\bQ(\zeta_q)^+ = \bQ(\zeta_q+\zeta_q^{-1})$ be the real subfield and $C^+_q \subset E^+_q$ the corresponding real subgroups. The cyclotomic units have finite index in the full unit group; unlike the full unit group, they enjoy a simple explicit basis \cite[Theorem 8.2]{Washington1997}:
\begin{prop}
For $1<a<q/2$, $(a,q)=1$, define
\[
\xi_a=\zeta_q^{(1-a)/2} \frac{1-\zeta_q^a}{1-\zeta_q}.
\]
The factor $\zeta_q^{(1-a)/2}$ is chosen so that $\xi_a \in \bR$. The $\xi_a$ and $-1$ form a basis for $C^+_q$ and we have
\[
\left[E_q^+:C_q^+ \right]=h_q^+
\]
where $h_q^+$ is the class number of $\bQ(\zeta_q)^+$.
\end{prop}

For any unit $u \in E_q^+$, we have $u^{h_q^+} \in C_q^+$ and $h(u^{h_q^+}) = h_q^+ h(u)$. Thus, we may count real units of height at most $B$ by counting real cyclotomic units of height at most $h_q^+ B$. Also, the units of $\bQ(\zeta_q)$ are those of $\bQ(\zeta_q)^+$ times a root of unity, so passing back to $\bQ(\zeta_q)$ we pick up a factor of $q$. Thus, to prove Theorem \ref{thm:main}, it suffices to prove the following:

\begin{prop}

    For all $\epsilon > 0$, we have $$U^+_{\cyc}(q,B) \ll q^2 \exp(C_{\epsilon} B^{1/3} q^{5/6+\epsilon}).$$
\end{prop}
\begin{proof}

Let $u$ be a real cyclotomic unit and write $u = \prod_i \xi_i^{c_i}$. For $1 \leq j < q/2$ with $(j,q) = 1$, let $\sigma_j$ be the element of the Galois group of $\bQ(\zeta_q)^+$ sending $\zeta_q+\zeta_q^{-1}$ to $\zeta_q^j+\zeta_q^{-j}$. Set $D=[\bQ(\zeta_q)^+ : \bQ]-1 = \phi(q)/2-1$. We write the Weil height
\[
h(u) = \frac{1}{[\bQ(\zeta_q)^+ : \bQ]} \sum_{\nu} \max(0, \log \abs{u}_\nu).
\]
Since $u$ is a unit, the non-archimedean valuations are trivial. The archimedean valuations are the absolute values of the Galois conjugates $\sigma_j(u)$. Thus, the height simplifies to
\[
h(u) = \frac{1}{D+1} \sum_j \max(0, \log \abs{\sigma_j(u)}).
\]
Furthermore, using $h(u) = h(u^{-1})$, we have
\begin{align*}
h(u) = \frac{h(u) + h(u^{-1})}{2} = \frac{1}{2(D+1)} \sum_{j} \abs{\log \abs{\sigma_j(u)}}.
\end{align*}
Form a $D \times D$ matrix $A$ with entries
\[
A_{ij} = \log \abs{\sigma_j(\xi_i)}
\]
for $1 < i, j < q/2$ with $(i,q)=(j,q)=1$. Then
\[
\norm{A c}_{L^1} = \sum_{j \neq 1} \abs{\log \abs{\sigma_j(u)}},
\]
and since $\sum_j \log \abs{\sigma_j(u)} = 0$ we have $2(D+1)h(u) \leq 2\norm{A c}_{L^1}$, so 
\[
\frac{1}{2(D+1)}
\leq \frac{h(u)}{\norm{A c}_{L^1}} \leq \frac{1}{D+1}.
\]
Define
\begin{align*}
N(q,T) :=
\# \{ c \in \bZ^D : \norm{A c}_{L^1} < T \}.
\end{align*}
Then $U^+_{\cyc}(q,B)$ is between $N(q,B(D+1))$ and $N(q,2B(D+1))$.

Now we can explicitly compute the action of $A$. For a nontrivial even Dirichlet character $\chi$ modulo $q$, define a vector $(v_\chi)_i = \chi(i)$.
\begin{prop}
    Let $\chi$ be a nontrivial even Dirichlet character modulo $q$, with conductor $f$, induced by $\chi_f$. Then
    \[
    A v_\chi = \lambda_\chi v_{\overline{\chi}}
    \]
    where
    \[
    \lambda_\chi
    =
    -\frac{1}{2}\tau(\chi_f)L(1,\overline{\chi_f}).
    \]
    In particular, for every $\epsilon > 0$,
    \[
    \abs{\lambda_\chi}
    =
    \frac{1}{2}\sqrt{f}\,\abs{L(1,\overline{\chi_f})}
    \geq C_\epsilon \sqrt{p}\, q^{-\epsilon}.
    \]
\end{prop}
\begin{proof}
First we record the following identity:
    \begin{align*}
    \sum_{(a,q) = 1} \chi(a) \log\abs{1-\zeta_q^a}
    =
    -\tau(\chi_f)
    L(1,\overline{\chi_f}).
    \end{align*}
    To derive this, write $q=p^r$ and $f=p^s$. We have
    \[
    \prod_{\substack{1 \leq a \leq p^r \\ a \equiv b \, (p^s)}} (1-\zeta_{p^r}^a)=1-\zeta_{p^s}^b,
    \]
    so
    \begin{align*}
    \sum_{\substack{1 \leq a \leq p^r \\ (a,p)=1}}\chi(a)\log\abs{1-\zeta_{p^r}^a}
    &=
    \sum_{\substack{1 \leq b \leq p^s \\ (b,p)=1}}
    \chi_f(b)
    \sum_{\substack{1 \leq a \leq p^r \\ a \equiv b \, (p^s)}}
    \log\abs{1-\zeta_{p^r}^a}
    \\
    &=
    \sum_{\substack{1 \leq b \leq p^s \\ (b,p)=1}}
    \chi_f(b)\log\abs{1-\zeta_{p^s}^b}.
    \end{align*}
    By \cite[Theorem 4.9]{Washington1997},
    \begin{align*}
    L(1,\overline{\chi_f})
    &=
    -\frac{\tau(\overline{\chi_f})}{f}
    \sum_{(b,f)=1}\chi_f(b)\log\abs{1-\zeta_f^b},
    \end{align*}
    and since $\chi_f$ is even, $\tau(\chi_f)\tau(\overline{\chi_f}) = \chi(-1)f = f$ and the identity follows.
Now we compute $(A v_\chi)_j$. Using evenness of $\chi$ we have
\begin{align*}
(A v_\chi)_j
&=
\sum_{\substack{1 < i < q/2 \\ (i,q) = 1}}
\chi(i)
\log \abs{\frac{1-\zeta_q^{ij}}{1-\zeta_q^j}}
\\
&=
\frac{1}{2}
\sum_{\substack{1 \leq i \leq q \\ (i,q) = 1}}
\chi(i)
\left[
\log\abs{1-\zeta_q^{ij}} - \log\abs{1-\zeta_q^j}
\right].
\end{align*}
where we also extend the range of the summation from $1 < i < q$ to $1 \leq i \leq q$ since the summand vanishes for $i = 1, q-1$. Since $\chi$ is nontrivial, the second term vanishes. Changing variables $k=ij$ in the first term yields
\begin{align*}
(A v_\chi)_j
&=
\frac{1}{2}
\overline{\chi(j)}
    \sum_{(k,q)=1}
    \chi(k)\log\abs{1-\zeta_q^k}
\\
&=
-\frac{1}{2}
\overline{\chi(j)}
\tau(\chi_f)L(1,\overline{\chi_f}),
\end{align*}
Thus $A v_\chi = \lambda_\chi v_{\overline{\chi}}$. Taking absolute values and using $\abs{\tau(\chi_f)}=\sqrt{f}$ gives
\[
\abs{\lambda_\chi}
=
\frac{1}{2}\sqrt{f}\,\abs{L(1,\overline{\chi_f})}.
\]
Finally, by Siegel's lower bound for Dirichlet $L$-functions \cite[Ch.~21]{Davenport2000}, $\abs{L(1,\chi_f)} \ge C_\epsilon f^{-\epsilon}$, so
\[
\abs{\lambda_\chi} \ge C_\epsilon \sqrt{p}\, q^{-\epsilon}.
\]
\end{proof}

Set
\[
N = D+1 = \phi(q)/2,
\qquad
G = (\bZ/q\bZ)^\times/\{\pm 1\}.
\]
For $a \in G \setminus \{1\}$, let $e_a$ be the basis vector with a $1$ at coordinate $a$ and $0$ elsewhere. Then
\[
e_a
=
\frac{1}{N}
\sum_{\chi \neq 1}
\left(\overline{\chi(a)}-1\right)v_\chi.
\]
The $v_\chi$ are not quite orthogonal, but we have
\[
\langle v_\chi,v_\psi\rangle
=
\sum_{b \in G \setminus \{1\}}\chi(b)\overline{\psi(b)}
=
N\delta_{\chi,\psi}-1.
\]
Therefore, for any coefficients $c_\chi$,
\[
\left\|\sum_{\chi \neq 1} c_\chi v_\chi\right\|_{2}^2
=
N\sum_{\chi \neq 1}\abs{c_\chi}^2
-
\left|\sum_{\chi \neq 1}c_\chi\right|^2
\leq
N\sum_{\chi \neq 1}\abs{c_\chi}^2.
\]
Using $A v_\chi=\lambda_\chi v_{\overline{\chi}}$ and relabeling $\chi$ by $\overline{\chi}$, we get
\[
A^{-1}e_a
=
\frac{1}{N}
\sum_{\chi \neq 1}
\frac{\chi(a)-1}{\lambda_\chi}v_\chi.
\]
Hence
\[
\norm{A^{-1} e_a}_{2}^2
\leq
\frac{1}{N}
\sum_{\chi \neq 1}
\frac{\abs{\chi(a)-1}^2}{\abs{\lambda_\chi}^2}.
\]
If $\chi$ has conductor $f_\chi = p^s$, then
\[
\frac{1}{\abs{\lambda_\chi}^2}
\ll_\epsilon
q^{2\epsilon}\frac{1}{p^s},
\]
so after grouping characters by conductor we get
\[
\norm{A^{-1} e_a}_{L^2}^2
\ll_\epsilon
\frac{q^{2\epsilon}}{N}
\sum_{s=1}^r
\frac{1}{p^s}
\sum_{f_\chi = p^s}
\abs{\chi(a)-1}^2
\ll
\frac{q^{2\epsilon}}{N}
\sum_{s=1}^r \frac{N_s}{p^s},
\]
where $N_s$ is the number of nontrivial even characters with conductor $p^s$. We have
\[
N_1 = \frac{p-3}{2},
\qquad
N_s = \frac{\phi(p^s)-\phi(p^{s-1})}{2}
=
\frac{p^{s-2}(p-1)^2}{2}
\quad (s \geq 2).
\]
Therefore $N_s/p^s \ll 1$, and since $r \ll \log q$, after renaming $\epsilon$ we have
\[
\norm{A^{-1} e_a}_{L^2} \ll_\epsilon q^{-1/2+\epsilon}.
\]
Let $\rho = 2B(D+1)$, and let $\cB_\rho$ be the $L^1$-ball of radius $\rho$ in $\bR^D$. We rewrite the problem as
\begin{align*}
N(q,2B(D+1)) = \#(\bZ^D \cap A^{-1} \cB_\rho).
\end{align*}
Now $A^{-1} \cB_\rho$ is a convex polytope spanned by the vertices $\pm \rho A^{-1} e_a$. The vertices have norm at most
\[
R = C_\epsilon B (D+1) q^{-1/2+\epsilon}.
\]
Now we are ready to apply the Davenport lemma. Let $\cP=A^{-1}\cB_\rho$. For a projection indexed by $I$, write $d = \abs{I}$. Each projection $\pi_I(\cP)$ is also a convex polytope, spanned by at most $2D$ vertices, and having circumradius at most $R$. To bound the volume, note that we can cover $\pi_I(\cP)$ with $\binom{2D}{d+1}$ simplices, defining each simplex by choosing a subset of $d+1$ vertices. The volume of a $d$-simplex with circumradius $R$ is maximized when the simplex is regular. The volume of a regular $d$-simplex is
\begin{align*}
R^d \frac{(d+1)^{(d+1)/2}}{d^{d/2} d!}
\ll
\frac{R^d}{(d-1)!}.
\end{align*}
Summing over all projections and all covering simplices gives
\begin{align*}
    \#(\bZ^D \cap \cP)
    &\leq
    1 + \sum_{d=1}^D
    \sum_{\substack{I \subset [D] \\ \abs{I} = d}}
    \operatorname{Vol}(\pi_I(\cP))
    \\
    &\ll
    1 + \sum_{d=1}^D
    \binom{D}{d} \binom{2D}{d+1}
    \frac{R^d}{(d-1)!}.
\end{align*}
Using \[
\binom{n}{k} \leq \left(\frac{en}{k} \right)^k,
\quad
n! \geq \left( \frac{n}{e} \right)^n
\]
we bound the log of the summand as follows.
\begin{align*}
\log \binom{D}{d} \binom{2D}{d+1}
    \frac{R^d}{(d-1)!} & \leq \log D + d \log D + d \log 2 D + d \log R - 3 d \log d + 3d
\\
&= \log D + d \log 2D^2R - 3 d \log d  + 3d.
\end{align*}
Differentiating and solving $\log(2D^2R) - 3 \log d = 0$, we see the maximum is achieved at $d^* = \min(D, (2D^2 R)^{1/3})$. If $R \geq D$, then $d^* = D$ and the log of the sum is
\begin{align*}
    &\leq 2\log D + D \log 2D^2 R - 3 D \log D + 3 D
    \\
    &\leq 2\log D + D \log (2R/D) + 3 D.
\end{align*}
Plugging in $R = C_\epsilon B D q^{-1/2+\epsilon}$ we have
\begin{align*}
N(q,2B(D+1)) &\ll  D^2 \left(C_\epsilon B q^{-1/2+\epsilon}\right)^D
\\
&\ll q^2 \left(C_\epsilon B q^{-1/2+\epsilon}\right)^{\phi(q)/2-1}.
\end{align*}
This is the expected asymptotic in the regime where the field is fixed and $B$ grows. On the other hand, unconditionally we have $d^* \leq (2D^2 R)^{1/3}$. The log of the sum is then
\begin{align*}
    &\leq 2\log D + d^* \log 2D^2R - 3 d^* \log d^*  + 3d^*
    \\
    &\leq 2\log D + 3(2D^2 R)^{1/3}.
\end{align*}
This gives
\begin{align*}
N(q,2B(D+1)) &\ll D^2 \exp \left( C_\epsilon B^{1/3} D q^{-1/6+\epsilon} \right)
\\
&\ll q^2 \exp \left( C_\epsilon B^{1/3} q^{5/6+\epsilon} \right).
\end{align*}

\end{proof}

We can also establish a lower bound for the number of cyclotomic units of bounded height by showing that the generators $\xi_a$ have height bounded by an absolute constant. Indeed,

\begin{prop}
    For sufficiently large $B$, $\log U_{\cyc}(q,B) \gg_B \log q$.
\end{prop}
\begin{proof}
    Recall that the height of $\xi_a$ is given by
\[
h(\xi_a) = \frac{1}{2(D+1)} \sum_{(j,q)=1} \abs{\log \abs{\frac{1-\zeta_q^{aj}}{1-\zeta_q^j}}}.
\]
By the triangle inequality,
\begin{align*}
h(\xi_a) &\leq \frac{1}{\phi(q)} \sum_{(j,q)=1} \left( \abs{\log \abs{1-\zeta_q^{aj}}} + \abs{\log \abs{1-\zeta_q^j}} \right) \\
&= \frac{2}{\phi(q)} \sum_{(j,q)=1} \abs{\log \abs{1-\zeta_q^j}} \\
&\ll \frac{1}{q-1}
\sum_{j=1}^{q-1} \abs{\log \abs{1-\zeta_q^j}}.
\end{align*}
Using the identity $\abs{1-\zeta_q^j} = 2\sin(\pi j/q)$, we see that this is a Riemann sum for the integral
\[
\int_0^1 \abs{\log(2\sin(\pi x))} \, dx,
\]
and it is easy to check that Riemann sums for this integral converge. Thus $h(\xi_a) \leq C$ for some absolute constant $C > 0$.

Now if $u = \prod_a \xi_a^{c_a}$ with $\sum_a \abs{c_a} \leq \lfloor B/C \rfloor$, we have $h(u) < B$. The number of such $c_a$ is given by the formula
\[
\sum_{k=0}^{\min(D,\lfloor B/C \rfloor)} 2^k \binom{D}{k} \binom{\lfloor B/C \rfloor}{k}.
\]
In particular, for $q$ sufficiently large in terms of $B$, taking the $k=\lfloor B/C \rfloor$ term gives $\gg_B q^{\lfloor B/C \rfloor}$. We obtain our theorem after taking logs.

\end{proof}

\subsection{Integers}
We now turn to counting integers. As with units, a polynomial lower bound is easy to see. To construct an integer of height $< B$, we may take any
\[
x = \sum_{i=0}^{D-1} c_i \zeta_q^i
\]
with $\sum_i \abs{c_i} < e^B$, then each conjugate has absolute value $< e^B$. If $n=\lceil e^B\rceil-1$, then like before there are
\[
\sum_{k=0}^{\min(D,n)} 2^k \binom{D}{k}\binom{n}{k} \gg_B q^n
\]
such choices of coefficients for $q$ sufficiently large in terms of $B$. We now prove the upper bound for integers.

\begin{proof}[Proof of Theorem \ref{thm:integers}]
Write $K=\bQ(\zeta_q)$ and $D=\varphi(q)$. Choose representatives $(\sigma_i)_{1 \leq i \leq D/2}$ for the Galois group up to conjugation, and define the real Minkowski embedding
\[
\Phi:K \longrightarrow \bR^D,
\qquad
\Phi(\alpha)=
\left(\operatorname{Re} \sigma_1(\alpha),\operatorname{Im} \sigma_1(\alpha),
\operatorname{Re} \sigma_2(\alpha),\operatorname{Im} \sigma_2(\alpha),\dots
\right).
\]
If $\alpha \in \bZ[\zeta_q]$, then only the archimedean places contribute to the Weil height, hence
\[
h(\alpha)
=
\frac{1}{D}\sum_{\tau:K \hookrightarrow \bC} \log^+\abs{\tau(\alpha)}
=
\frac{1}{D}
\sum_{i=1}^{D/2}
\log^+\left((\operatorname{Re}\sigma_i(\alpha))^2+(\operatorname{Im}\sigma_i(\alpha))^2\right).
\]
Define
\[
R_1=
\left\{
x \in \bR^D :
\sum_{i=1}^{D/2}
\log^+\left(x_{2i-1}^2+x_{2i}^2\right)
\leq DB
\right\},
\]
Then
\[
I(q,B)=\#(\Phi(\bZ[\zeta_q]) \cap R_1).
\]
Now enlarge $R_1$ by replacing the disk at each complex place by a square:
\[
R_2=
\left\{
x \in \bR^D :
\sum_{i=1}^{D/2}
\log^+\max\bigl(\abs{x_{2i-1}}^2,\abs{x_{2i}}^2\bigr)
\leq DB
\right\}.
\]
This clearly contains $R_1$, so
\[
I(q,B)\leq \#(\Phi(\bZ[\zeta_q]) \cap R_2).
\]
Set
\[
L=\left\lceil \frac{DB}{\log 2} \right\rceil + \frac{D}{2},
\]
and
\[
\mathcal K=
\left\{
k=(k_i)_{1 \leq i \leq D/2} \in \bZ_{\geq 0}^{D/2} :
\sum_{i=1}^{D/2} k_i=L
\right\}.
\]
For each $k \in \mathcal K$, we define a box
\[
Q(k)=
\left\{
x \in \bR^D :
\abs{x_{2i-1}}, \abs{x_{2i}} \leq 2^{k_i/2}
\text{ for all } 1 \leq i \leq D/2
\right\}.
\]
All these boxes have the same real volume
\[
\Vol(Q(k))=\prod_{i=1}^{D/2} \left(2^{1+k_i/2}\right)^2 = 2^{L+D}.
\]
Now let $x \in R_2$ and define
\[
\kappa_i=
\max\left(0,\left\lceil
\log_2 \max\bigl(\abs{x_{2i-1}}^2,\abs{x_{2i}}^2\bigr)
\right\rceil\right).
\]
Then $\sum_i \kappa_i \leq L$. If we choose $k \in \mathcal K$ with $k_i \geq \kappa_i$ for all $i$, then $x \in Q(k)$. Thus the family of boxes $Q(k)$ covers $R_2$.
Taking a maximum over all boxes we have
\[
I(q,B)
\ll
\abs{\mathcal K}
\max_k \#(\Phi(\bZ[\zeta_q]) \cap Q(k)).
\]
By stars and bars, the number of such $k$ is
\[
\abs{\mathcal K}=\binom{L+D/2-1}{D/2-1} \leq \exp(C_1 D(B+1)).
\]
Write
\[
Q(k)=\prod_{t=1}^D [-\ell_t/2,\ell_t/2] \subset \bR^D.
\]
For the two real coordinates attached to a fixed $\sigma_i$, the corresponding side lengths are both equal to $2^{1+k_i/2}$.

Now let $M$ be the $D \times D$ matrix defined by
\begin{align*}
M_{2j-1,i} &= \operatorname{Re}(\sigma_j \zeta_q^{i-1})
\\
M_{2j,i} &= \operatorname{Im}(\sigma_j \zeta_q^{i-1})
\end{align*}
for $1 \leq i \leq D$ and $1 \leq j \leq D/2$. Then $\Phi(\bZ[\zeta_q])=M\bZ^D$, so
\[
\#(\Phi(\bZ[\zeta_q]) \cap Q)
=
\#(\bZ^D \cap M^{-1}Q).
\]
Applying Davenport's lemma gives
\[
\#(\bZ^D \cap M^{-1}Q)
\leq
1+\sum_{\emptyset \neq I \subset [D]} h^{D-\abs{I}} \Vol(\pi_I(M^{-1}Q)),
\]
and since $M^{-1}Q$ is convex, every line meets $\pi_I(M^{-1} Q)$ in at most one interval, so we may set $h=1$.Fix $j$ and sum over all $I \subset [D]$ with $\abs{I}=j$. Since $Q$ is a box, each $\pi_I(M^{-1}Q)$ is a zonotope, and the zonotope volume formula gives
\[
\Vol(\pi_I(M^{-1}Q))
\leq
\sum_{\substack{J \subset [D] \\ \abs{J}=j}}
\abs{[M^{-1}]_{I,J}}
\prod_{t \in J} \ell_t.
\]
Hence by Cauchy-Schwarz,
\begin{align*}
\sum_{\substack{I \subset [D] \\ \abs{I}=j}} \Vol(\pi_I(M^{-1}Q))
&\leq
\sum_{\substack{I,J \subset [D] \\ \abs{I}=\abs{J}=j}}
\abs{[M^{-1}]_{I,J}}
\prod_{t \in J} \ell_t \\
&\leq
\left(
\sum_{\substack{I,J \subset [D] \\ \abs{I}=\abs{J}=j}}
\abs{[M^{-1}]_{I,J}}^2
\right)^{1/2}
\left(
\sum_{\substack{I,J \subset [D] \\ \abs{I}=\abs{J}=j}}
\prod_{t \in J} \ell_t^2
\right)^{1/2}.
\end{align*}
For $0 \leq j \leq D$, define
\[
E_j(Q)=\sum_{\substack{J \subset [D] \\ \abs{J}=j}} \prod_{t \in J} \ell_t.
\]
The second factor is
\[
\binom{D}{j}^{1/2} \left( \sum_{\substack{J \subset [D] \\ \abs{J}=j}} \prod_{t \in J} \ell_t^2 \right)^{1/2}
\leq
\binom{D}{j}^{1/2} E_j(Q),
\]
since $\sqrt{\sum x_i^2} \leq \sum x_i$.
Also,
\[
\sum_{j=0}^D E_j(Q)=\prod_{t=1}^D (1+\ell_t)
\ll
\prod_{i=1}^{D/2} (1+C 2^{k_i/2})^2
\ll
C_2^D 2^{\sum_i k_i}
\leq
\exp(C_3 D(B+1)).
\]

For the first factor, Jacobi's complementary minor identity gives
\[
[M^{-1}]_{I,J}
=
\pm \frac{[M]_{J^c,I^c}}{\det(M)}.
\]
Therefore
\[
\sum_{\substack{I,J \subset [D] \\ \abs{I}=\abs{J}=j}}
\abs{[M^{-1}]_{I,J}}^2
=
\frac{1}{\abs{\det(M)}^2}
\sum_{\substack{I',J' \subset [D] \\ \abs{I'}=\abs{J'}=D-j}}
\abs{[M]_{I',J'}}^2.
\]
By Cauchy-Binet, for each fixed $I' \subset [D]$ with $\abs{I'}=D-j$ we have
\[
\sum_{\substack{J' \subset [D] \\ \abs{J'}=D-j}}
\abs{[M]_{I',J'}}^2
=
[MM^T]_{I',I'}.
\]
Therefore
\[
\sum_{\substack{I',J' \subset [D] \\ \abs{I'}=\abs{J'}=D-j}}
\abs{[M]_{I',J'}}^2
=
\sum_{\substack{I' \subset [D] \\ \abs{I'}=D-j}}
[MM^T]_{I',I'}.
\]
Each row of $M$ has Euclidean norm at most $\sqrt{D}$, so every diagonal entry of $MM^T$ is at most $D$. Hence Hadamard's inequality gives
\[
[MM^T]_{I',I'} \leq D^{D-j},
\]
and therefore
\[
\sum_{\substack{I',J' \subset [D] \\ \abs{I'}=\abs{J'}=D-j}}
\abs{[M]_{I',J'}}^2
\leq
\binom{D}{j} D^{D-j}.
\]
Thus
\[
\sum_{\substack{I \subset [D] \\ \abs{I}=j}} \Vol(\pi_I(M^{-1}Q))
\leq
\binom{D}{j}\frac{D^{(D-j)/2}}{\abs{\det(M)}} E_j(Q).
\]

Now use the cyclotomic discriminant formula
\[
\abs{\disc(K)}=\frac{q^D}{p^{D/(p-1)}},
\]
so
\[
\abs{\det(M)}=2^{-D/2}\abs{\disc(K)}^{1/2}
=
2^{-D/2} q^{D/2} p^{-D/(2(p-1))}.
\]
Since $D=q(1-1/p)$, we obtain
\[
\frac{D^{D/2}}{\abs{\det(M)}}
=
\left( \sqrt{2}\sqrt{1-\frac{1}{p}}\, p^{1/(2(p-1))} \right)^D
\leq
C_4^D,
\]
and therefore
\[
\frac{D^{(D-j)/2}}{\abs{\det(M)}}
\leq
\frac{D^{D/2}}{\abs{\det(M)}}
\leq
C_4^D.
\]
It follows that
\[
\#(\bZ^D \cap M^{-1}Q)
\leq
1+C_4^D\sum_{j=1}^D \binom{D}{j} E_j(Q)
\leq
C_5^D \sum_{j=0}^D E_j(Q)
\leq
\exp\left(C_6 D(B+1)\right).
\]

Finally, summing over all dyadic boxes gives
\[
\log I(q,B) \ll D(B+1).
\]
\end{proof}

\bibliographystyle{abbrv}
\bibliography{references}

\end{document}

%% file: symbols.tex
\newtheorem{prop}{Proposition}

\newtheorem{thm}{Theorem}

\newcommand{\bC}{\mathbb{C}}

\newcommand{\bF}{\mathbb{F}}

\newcommand{\bQ}{\mathbb{Q}}
\newcommand{\bR}{\mathbb{R}}

\newcommand{\bZ}{\mathbb{Z}}

\DeclareMathOperator{\Vol}{Vol}
\DeclareMathOperator{\disc}{disc}

\DeclareMathOperator{\cyc}{cyc}

\def\bbF{ {\mathbb F}}
\def\bbQ{ {\mathbb Q}}
\def\bbZ{ {\mathbb Z}}
\def\bbP{ {\mathbb P}}

\def\bb1{ {\mathbb 1}}

\def\cB{ {\mathcal B} }

\def\cP{{\mathcal P}}

